\documentclass[12pt]{amsart}
\usepackage[shortlabels]{enumitem}
\usepackage{amsmath, amsthm, amsfonts, amssymb, mathrsfs, graphicx, stmaryrd}
\usepackage[all]{xy}
\usepackage[usenames, dvipsnames]{color}
\usepackage{xcolor}
\usepackage[margin=1in]{geometry} 
\usepackage[bookmarks, bookmarksdepth=2, colorlinks=true, linkcolor=blue, citecolor=blue, urlcolor=blue]{hyperref}
\usepackage{eucal}
\usepackage{comment}

\newcommand{\bA}{\mathbf{A}}

\newcommand{\bC}{\mathbf{C}}
\newcommand{\cC}{\mathcal{C}}

\newcommand{\cF}{\mathcal{F}}

\newcommand{\rH}{\mathrm{H}}

\newcommand{\cK}{\mathcal{K}}

\newcommand{\fM}{\mathfrak{M}}

\newcommand{\bN}{\mathbf{N}}

\newcommand{\cP}{\mathcal{P}}

\newcommand{\bQ}{\mathbf{Q}}
\newcommand{\cQ}{\mathcal{Q}}

\newcommand{\bS}{\mathbf{S}}

\newcommand{\fS}{\mathfrak{S}}

\newcommand{\rT}{\mathrm{T}}

\newcommand{\bV}{\mathbf{V}}
\newcommand{\cV}{\mathcal{V}}

\newcommand{\bW}{\mathbf{W}}
\newcommand{\cW}{\mathcal{W}}
\newcommand{\fW}{\mathfrak{W}}

\newcommand{\bZ}{\mathbf{Z}}

\newcommand{\fa}{\mathfrak{a}}

\newcommand{\fh}{\mathfrak{h}}

\newcommand{\bk}{\mathbf{k}}

\newcommand{\rn}{\mathrm{n}}

\newcommand{\ru}{\mathrm{u}}

\let\lbb\llbracket
\let\rbb\rrbracket

\newcommand{\arxiv}[1]{\href{http://arxiv.org/abs/#1}{{\tiny\tt arXiv:#1}}}
\newcommand{\DOI}[1]{\href{http://doi.org/#1}{\color{purple}{\tiny\tt DOI:#1}}}

\numberwithin{equation}{section}
\newtheorem{theorem}[equation]{Theorem}

\newtheorem{proposition}[equation]{Proposition}
\newtheorem{lemma}[equation]{Lemma}
\newtheorem{corollary}[equation]{Corollary}

\theoremstyle{definition}
\newtheorem{rmk}[equation]{Remark}
\newenvironment{remark}[1][]{\begin{rmk}[#1] \pushQED{\qed}}{\popQED \end{rmk}}
\newtheorem{eg}[equation]{Example}
\newenvironment{example}[1][]{\begin{eg}[#1] \pushQED{\qed}}{\popQED \end{eg}}
\newtheorem{defnux}[equation]{Definition}
\newenvironment{definition}[1][]{\begin{defnux}[#1]\pushQED{\qed}}{\popQED \end{defnux}}

\newcommand{\defn}[1]{\emph{#1}}

\renewcommand{\phi}{\varphi}
\renewcommand{\emptyset}{\varnothing}
\newcommand{\lw}{{\textstyle \bigwedge}}
\DeclareMathOperator{\im}{im} 

\DeclareMathOperator{\coker}{coker}

\DeclareMathOperator{\End}{End}
\DeclareMathOperator{\Sym}{Sym}

\DeclareMathOperator{\Hom}{Hom}

\newcommand{\id}{\mathrm{id}}
\newcommand{\op}{\mathrm{op}}
\newcommand{\pol}{\mathrm{pol}}
\renewcommand{\Vec}{\mathrm{Vec}}

\newcommand{\fgl}{\mathfrak{gl}}
\newcommand{\FS}{\mathbf{FS}}
\newcommand{\FB}{\mathbf{FB}}

\DeclareMathOperator{\Com}{Com}
\newcommand{\NCFin}{\mathbf{NCFin}}
\newcommand{\Pol}{\mathbf{Pol}}

\DeclareMathOperator{\As}{As}
\DeclareMathOperator{\Alg}{Alg}
\newcommand{\Fin}{\mathbf{Fin}}

\DeclareMathOperator{\Rep}{Rep}

\newcommand{\inv}{^{-1}}
\newcommand{\del}{\partial}

\newcommand{\aut}{\mathfrak{aut}}
\DeclareMathOperator{\ch}{ch}
\newcommand{\pos}{+}

\title{Polynomial representations of the Witt Lie algebra}

\author{Steven V Sam}
\address{Department of Mathematics, University of California, San Diego, CA, USA}
\email{\href{mailto:ssam@ucsd.edu}{ssam@ucsd.edu}}
\urladdr{\url{http://math.ucsd.edu/~ssam/}}
\thanks{SS was supported by NSF grant DMS-1849173.}

\author{Andrew Snowden}
\address{Department of Mathematics, University of Michigan, Ann Arbor, MI, USA}
\email{\href{mailto:asnowden@umich.edu}{asnowden@umich.edu}}
\urladdr{\url{http://www-personal.umich.edu/~asnowden/}}
\thanks{AS was supported by NSF grants DMS-1303082 and DMS-1453893 and a Sloan Fellowship.}

\author{Philip Tosteson}
\address{Department of Mathematics, University of Chicago, Chicago, IL, USA}
\curraddr{Department of Mathematics, University of North Carolina, Chapel Hill, NC, USA}
\email{\href{mailto:ptoste@unc.edu}{ptoste@unc.edu}}
\urladdr{\url{https://ptoste.github.io/}}
\thanks{PT was supported by NSF grant DMS-1903040.}

\date{March 18, 2024}

\begin{document}

\begin{abstract}
The \defn{Witt algebra} $\fW_n$ is the Lie algebra of all derivations of the $n$-variable polynomial ring $\bV_n=\bC[x_1, \ldots, x_n]$ (or of algebraic vector fields on $\bA^n$). A representation of $\fW_n$ is \defn{polynomial} if it arises as a subquotient of a sum of tensor powers of $\bV_n$.  Our main theorems assert that finitely generated polynomial representations of $\fW_n$ are noetherian and have rational Hilbert series.  A key intermediate result  states polynomial representations of the infinite Witt algebra are equivalent to representations of $\Fin^{\op}$, where $\Fin$ is the category of finite sets.  We also show that polynomial representations of $\fW_n$ are equivalent to polynomial representations of the endomorphism monoid of $\bA^n$.  These equivalences are a special case of an operadic version of Schur--Weyl duality,  which we establish.  \end{abstract}

\newpage

\maketitle
\tableofcontents

\section{Introduction}

The Witt algebra $\fW_n$ is an important infinite dimensional Lie algebra in mathematics and physics. The purpose of this paper is to investigate a class of representations of $\fW_n$ called the \emph{polynomial representations}, which are related to representations of the monoid of algebraic endomorphisms of $\bA^n$.  From the point of view of classical representation theory, these representations are a bit unusual: the main representations of interest are finitely generated but of infinite length, while the structure of irreducible representations is not so complicated. Thus this theory is more like the module theory of a commutative ring than a typical representation theory. Accordingly, our main results mirror basic results from commutative algebra: we show that finitely generated polynomial representations of $\fW_n$ are noetherian and have rational Hilbert series. In the rest of the introduction, we explain our results and motivation in more detail.

\subsection{Definitions}

Fix a field $\bk$ of characteristic~0. The \defn{Witt algebra} $\fW_n$ is the Lie algebra of derivations of the polynomial ring $\bV_n=\bk[x_1, \ldots, x_n]$. Derivations of the form $f \frac{\partial}{\partial x_i}$, with $f \in \bV_n$, span $\fW_n$ as a vector space. The algebra $\fW_n$ is graded, with $f \frac{\partial}{\partial x_i}$ being homogeneous of degree $d-1$ if $f$ is homogeneous of degree $d$. The grading is supported in degrees $\ge -1$. The degree~0 piece of $\fW_n$ is a Lie subalgebra isomorphic to $\fgl_n$. We let $\fW^{\pos}_n$ be the sum of the non-negative pieces of $\fW_n$, which is also a Lie subalgebra.

The vector space $\bV_n$ carries a tautological representation of $\fW_n$, and we refer to $\bV_n$ as the \defn{standard representation} of $\fW_n$. We say that a representation of $\fW_n$ is \defn{polynomial} if it can be realized as a subquotient of a (possibly infinite) direct sum of tensor powers of the standard representation. We let $\Rep^{\pol}(\fW_n)$ denote the category of polynomial representations. It is a Grothendieck abelian category that is closed under tensor products. We emphasize that a finitely generated polynomial representation of $\fW_n$ need not be finite dimensional, or even of finite length (see Example~\ref{ex:not-fin-len}).

\begin{example}
The algebra $\fW_n$ naturally acts on the module of K\"ahler differentials of $\bV_n$, and this is a polynomial representation. See \S \ref{ss:kahler} for more details.
\end{example}

\begin{example} \label{ex:not-fin-len}
Let $\fa_n$ be the ideal of $\bV_1^{\otimes 2}=\bk[x,y]$ generated by $(x-y)^n$. Then the $\fa_n$ form a strictly descending chain of $\fW_1$-submodules. This shows that $\bV_1^{\otimes 2}$ does not have finite length. However, one can show that $\bV_1^{\otimes 2}$ is finitely generated as a $\fW_1$-module.
\end{example}

\subsection{Main results}

We now state our two main theorems. The first is an analog of the Hilbert basis theorem:

\begin{theorem} \label{mainthm1}
If $M$ is a finitely generated polynomial representation of $\fW_n$ then any subrepresentation of $M$ is also finitely generated.
\end{theorem}

In other words, this theorem states that the category $\Rep^{\pol}(\fW_n)$ is locally noetherian. In contrast, Sierra--Walton \cite{sw1} have shown that the category of \emph{all} $\fW_n$-modules is not locally noetherian.

To state our second theorem, we must first make a definition. Let $M$ be a finitely generated polynomial $\fW_n$-module. Via the inclusion $\fgl_n \subset \fW_n$, we can regard $M$ as a representation of $\fgl_n$. As such, it is a polynomial representation; this follows since $\bV_n$ is a polynomial representation of $\fgl_n$. It therefore carries a natural grading (induced by the action of the center of $\fgl_n$) supported in non-negative degrees. We define the \defn{Hilbert series} of $M$ to be the formal series
\begin{displaymath}
\rH_M(t) = \sum_{n \ge 0} \dim{M_n} \cdot t^n.
\end{displaymath}
It is not difficult to see that $M_n$ is in fact finite dimensional when $M$ is finitely generated. Our second result concerns this invariant:

\begin{theorem} \label{mainthm2}
If $M$ is a finitely generated polynomial representation of $\fW_n$ then $\rH_M(t)$ is a rational function of $t$; moreover, its denominator is a product of factors of the form $1-t^m$.
\end{theorem}

Equivalently, the theorem asserts that the function $n \mapsto \dim{M_n}$ is a quasi-polynomial of $n$, at least when $n$ is sufficiently large. We also prove a more precise result about how the character of $M_n$ (as a $\fgl_n$-module) varies with $n$. We note that the above results apply equally well to $\fW_n^{\pos}$.

\subsection{Overview of proof}

To prove our results, we pass to the infinite Witt algebra $\fW=\bigcup_{n \ge 1} \fW_n$. This acts on $\bV=\bigcup_{n \ge 1} \bV_n$ (the infinite variable polynomial ring), which leads to a similar notion of polynomial representation. We show that there is a \defn{specialization functor}
\begin{displaymath}
\Gamma_n \colon \Rep^{\pol}(\fW) \to \Rep^{\pol}(\fW_n),
\end{displaymath}
which allows us to transfer information from $\fW$ to $\fW_n$. In fact, $\Gamma_n$ is a Serre quotient; see Proposition~\ref{prop:serre}.

The following is the key theorem that gives us a handle on $\fW$:

\begin{theorem} \label{mainthm3}
The category $\Rep^{\pol}(\fW)$ is equivalent to the category of $\Fin^{\op}$-modules.
\end{theorem}

Here $\Fin$ is the category of finite sets, and a $\Fin^{\op}$-module is a functor $\Fin^{\op} \to \Vec$, where $\Vec$ is the category of vector spaces. To prove Theorem~\ref{mainthm3}, we show that the representations $\bV^{\otimes n}$ of $\fW$ are projective (in the category of polynomial representations), and that this category of projectives is equivalent to the category of principal projective $\Fin^{\op}$-modules.

Using Theorem~\ref{mainthm3} we can transfer known results about $\Fin^{\op}$-modules to $\fW$. In particular, a theorem of Sam--Snowden \cite{catgb} states that the category of $\Fin^{\op}$-modules is locally noetherian. We thus see that $\Rep^{\pol}(\fW)$ is locally noetherian, and (via properties of specialization) this yields Theorem~\ref{mainthm1}. Similarly, results of Tosteson \cite{To} give control on the Hilbert series of $\Fin^{\op}$-modules, and this allows us to deduce Theorem~\ref{mainthm2}.

\begin{remark} \label{rmk:othercat}
In fact, there are two additional categories that are equivalent to the two categories in Theorem~\ref{mainthm3}. The first is the category of polynomial representations of the monoid $\fM$ of all algebra endomorphisms of $\bV$,  equivalently  the monoid of endomorphisms of $\bA^\infty$. (This monoid plays the role of a ``Lie group'' for $\fW$). The second is the category of polynomial $\cC$-modules, where $\cC$ is the category of finitely generated polynomial rings. See \S \ref{ss:setup} for precise definitions.
\end{remark}

\begin{remark} \label{rmk:mainthm3}
There is a version of Theorem~\ref{mainthm3} for $\fW^{\pos}$, but here $\Fin$ is replaced by $\FS$, the category of finite sets and surjective functions.
\end{remark}

\begin{remark} \label{rmk:curried}
Theorem~\ref{mainthm3} is studied in \cite{brauercat2} from the perspective of ``curried algebra.''\footnote{We note that while the preprint \cite{brauercat2} appeared publicly before this paper, the discussion of the Witt algebra in \cite{brauercat2} was inspired by this project.}
\end{remark}

\subsection{An operadic generalization}

Theorem~\ref{mainthm3} is in fact a special case of a much more general theorem. Let $P$ be a linear operad (see \S \ref{ss:operad}). Let $\bV$ be the free $P$-algebra on a countable set of generators, and let $\fh$ be the Lie algebra of derivations of $\bV$. We obtain a notion of polynomial representation for $\fh$ from its tautological action on $\bV$. Let $\cW$ be the wiring category (or PROP) associated to $P$ (see \S \ref{ss:wiring}); this is a $\bk$-linear category that is typically of a combinatorial flavor. We show the following:

\begin{theorem} \label{mainthm4}
The category $\Rep^{\pol}(\fh)$ is equivalent to the category of $\cW^{\op}$-modules.
\end{theorem}

Here are a few special cases of this theorem:
\begin{itemize}
\item Suppose $P$ is the trivial operad (see \S \ref{sss:trivial}). Then $\fh=\fgl$ is the infinite general linear Lie algebra and $\cW=\bk[\FB]$ is (the linearization of) the category of finite sets and bijections. A $\cW^{\op}$-module is a linear species, i.e., a sequence of representations of symmetric groups. Theorem~\ref{mainthm4} in this case reduces to classical Schur--Weyl duality.
\item Suppose $P=\Com$ is the unital commutative operad (see \S \ref{sss:com}). Then $\fh=\fW$ and $\cW=\bk[\Fin]$, and Theorem~\ref{mainthm4} is exactly Theorem~\ref{mainthm3}.
\item Suppose $P=\Com^{\rn\ru}$ is the non-unital commutative operad (see \S \ref{sss:comnu}). Then $\fh=\fW^{\pos}$ and $\cW=\bk[\FS]$, and Theorem~\ref{mainthm4} is the analog of Theorem~\ref{mainthm3} mentioned in Remark~\ref{rmk:mainthm3}.
\end{itemize}
The proof of Theorem~\ref{mainthm4} follows the outline given for the proof of Theorem~\ref{mainthm3}, i.e., we match up categories of projective objects. Remark~\ref{rmk:othercat} also applies in the setting of a general operad; see \S \ref{ss:setup} for details.

\subsection{Motivation}

Recall that $\FS$ is the category of finite sets and surjections and an $\FS^{\op}$-module is a functor $\FS^{\op} \to \Vec$. These have seen several recent applications.
\begin{itemize}
\item Snowden \cite{delta-mod} proved a finiteness result for syzygies of the Segre embedding using certain algebraic objects called \defn{$\Delta$-modules}, which are closely related to $\FS^{\op}$-modules; in fact, \cite{catgb} proved stronger results on $\Delta$-modules by leveraging its new results on $\FS^{\op}$-modules. Draisma--Kuttler \cite[\S 7]{DraismaKutler} also discuss the \defn{substitution monoid} in this context, which is again closely related to $\FS^{\op}$.
\item Sam--Snowden \cite{catgb} proved the Lannes--Schwartz artinian conjecture using $\FS^{\op}$-modules. (We note this conjecture was also proved by Putman--Sam \cite{PutmanSam} around the same time with slightly different methods.)
\item Tosteson \cite{Mgnbar,  Kontsevich} showed that the homology of moduli spaces of stable curves admits the structure of a finitely generated $\FS^\op$-module, implying that these homology groups exhibit a form of representation stability.  
\end{itemize}
We are therefore interested in understanding the structure of $\FS^{\op}$-modules in more detail. This seems quite difficult, and we have made little progress by direct attack.

When we realized the connection between $\FS^{\op}$-modules and $\fW^{\pos}$-modules, we were hopeful that we might be able to use the latter perspective to shed light on $\FS^{\op}$-modules. In this paper, the information actually flows in the other direction: we translate results from $\FS^{\op}$ to $\fW^{\pos}$. However, we are still hopeful that certain problems might be easier on the $\fW^{\pos}$ side, especially because of the more geometric nature of $\fW^{\pos}$. For example, a positive solution to the problem \S \ref{ss:problems}(b) would likely lead to new results on $\FS^{\op}$-modules.

There is also motivation purely on the Witt algebra side. Indeed, there are some open finiteness questions about the Witt algebra (e.g., \cite[Conjecture~1.2]{PetukhovSierra} and \cite[Conjecture~1.3]{PetukhovSierra}) to which the results or methods of this paper could be relevant.  

 In addition, there is a conjecture of Gelfand concerning the finite dimensionality of the homology of subalgebras of the Witt Lie algebra.  After we posted this paper,  Feigin--Kanel-Belov--Khoroshkin \cite{FKK} showed that Gelfand's conjecture follows from a noetherian property similar to the one here.  (The proof of the noetherian property in \cite{FKK} contains a gap, but their paper will be updated soon.)

\subsection{Open problems} \label{ss:problems}

We mention a few open problems in connection to this work:
\begin{enumerate}
\item Can Theorem~\ref{mainthm1} or~\ref{mainthm2} be proved directly, i.e., without passing through $\Fin^{\op}$-modules?  Even for $\fW_1$, we do not know a proof that avoids considering the infinite Witt algebra $\fW$.
\item Let $R_d$ be the $\fS_d$-invariant subalgebra of $\bV_1^{\otimes d} \cong \bk[x_1, \ldots, x_d]$. The ring $R_d$ carries an action of $\fW_1^{\pos}$ by algebra derivations. Let $M$ be a finitely generated $R_d$-module equipped with a compatible polynomial action of $\fW_1^{\pos}$. Is it true that any non-zero $\fW_1^{\pos}$-submodule of $M$ contains a non-zero $R_d$-submodule? We can prove this for $M=R_d$. One can pose a similar question for $\fW_n^{\pos}$ with $n>1$. A positive answer to this question would lead to much finer results on the structure of polynomial $\fW_n^{\pos}$-modules.
\item Let $\bV_n^{\vee}$ be the graded dual of $\bV_n$. Define a $\fW_n$-module to be \defn{algebraic} if it belongs to the tensor category generated by $\bV_n$ and $\bV_n^{\vee}$. The adjoint representation of $\fW_1$ is an example of an algebraic (and not polynomial) representation; see \S \ref{ss:finite-ex}. Is there a noetherian result for algebraic representations? We note that Petukhov--Sierra \cite[Theorem~4.2]{PetukhovSierra} have shown that $\Sym^2(\fW^{>0}_1)$ is noetherian, where $\fW^{>0}_1$ denotes the strictly positive part of $\fW_1$.
\item What is the Krull--Gabriel dimension of the $\fW_n$-module $\bV_n^{\otimes d}$?
\end{enumerate}

\subsection{Relation to other work}

Sierra--Walton \cite{sw1,sw2} studied the universal enveloping algebra of $\fW_n$, and showed that it is not noetherian. In the opposite direction, Petukhov--Sierra \cite{PetukhovSierra} proved that $\Sym^2(\fW^{>0}_1)$ is noetherian as a $\fW^{>0}_1$-module.

The Witt algebra $\fW_n$, was studied by Feigin--Shoiket \cite{FS} in connection to a computation of the  central series of  associative algebras. Feigin--Shoiket classified the irreducible representations of $\fW_n$ satisfying a semisimplicity and nonnegativity condition on weight spaces. In contrast to \cite{FS}, our results concern the Hilbert series of $\fW_n$ representations  with infinite length.

Our results imply that the category of finite length polynomial representations of $\fW_n$ is contravariantly equivalent to the category of finitely generated $\Fin$-modules. Wiltshire-Gordon classified the simple $\Fin$-modules \cite{wiltshire-gordon}, and under this equivalence, Feigin--Shoiket's classification of irreducibles corresponds precisely to Wiltshire-Gordon's.

Fresse \cite{Fresse} has used modules over an operad $P$ in the study of the  homological algebra of functors from $P$-algebras to  other categories. Powell \cite{Powell} has studied modules for the wiring category of an operad (and the Lie operad specifically).

Finally, we mention \cite{nishiyama}, which studies the finite-dimensional version of Schur--Weyl duality for the Witt Lie algebra.

\subsection{Outline}

In \S \ref{s:bg} we review operads, wiring categories, and the classical polynomial representation theory of $\fgl_n$. In \S \ref{s:equiv} we prove Theorem~\ref{mainthm4} relating the Lie algebra $\fh$ in infinite dimensions to the wiring category $\cW$. In \S \ref{s:findim}, we discuss the specialization functor, which relates $\fh$ to a Lie algebra $\fh_n$ in finite dimensions. Theorem~\ref{mainthm1} is proved here. Finally, in \S \ref{s:hilbert} we discuss Hilbert series, and prove Theorem~\ref{mainthm2}.

\subsection*{Acknowledgements}

PT would like to thank John Wiltshire--Gordon for introducing him to wiring categories of operads, and advocating for their use in representation stability.  

\section{Background} \label{s:bg}

In this section, we review some background material on polynomial representations of the general linear group, operads, and wiring categories.

\subsection{Polynomial representations} \label{ss:poly}

Fix a field $\bk$ of characteristic zero. All vector spaces are over $\bk$. A \defn{representation} of a category $\cC$ (or a \defn{$\cC$-module}) is a functor $\cC \to \Vec$; when $\cC$ is a $\bk$-linear category, we require representations to be $\bk$-linear functors. We write $\Rep(\cC)$ for the category of representations.
 
Let $\FB$ be the category of finite sets and bijections. For two $\FB$-modules $M$ and $N$, we define an $\FB$-module $M \otimes N$ by 
\[
  (M \otimes N)(S) = \bigoplus_{S=A \sqcup B} M(A) \otimes N(B)
\]
This construction defines a symmetric monoidal structure on $\Rep(\FB)$.

Let $\Pol$ be the category of polynomial functors $\Vec \to \Vec$. Since we are in characteristic~0, every polynomial functor is isomorphic to a (perhaps infinite) direct sum of Schur functors (see \cite[\S 5.3]{expos} for more information and references on Schur functors). Given an $\FB$-module $M$ and a vector space $V$, put
\[
  M\{V\} := \bigoplus_{r \ge 0} M([r]) \otimes_{\bk[\fS_r]} V^{\otimes r},
\]
where $\fS_r$ acts on $V^{\otimes r}$ by permuting tensor factors. Then $M\{-\}$ is a polynomial functor and the construction $M \mapsto M\{-\}$ defines an equivalence of symmetric monoidal categories $\Rep(\FB) \to \Pol$. (The tensor product of polynomial functors is defined pointwise.)

Let $\fgl=\bigcup_{n \ge 1} \fgl_n$ be the infinite general linear Lie algebra. This has a standard representation on $\bW=\bigcup_{n \ge 1} \bk^n$. We say that a representation of $\fgl$ is \defn{polynomial} if it is isomorphic to a direct sum of Schur functors applied to $\bW$. We write $\Rep^{\pol}(\fgl)$ for the category of polynomial representations. The evaluation functor
\begin{displaymath}
\Pol \to \Rep^{\pol}(\fgl), \qquad F \mapsto F(\bW)
\end{displaymath}
is an equivalence of categories, and compatible with the symmetric monoidal structures.

A \defn{weight} of $\fgl$ is a sequence $\lambda=(\lambda_1, \lambda_2, \ldots)$ of integers, almost all of which are zero. The \defn{support} of $\lambda$ is the set of indices $i$ for which $\lambda_i \ne 0$. We define the \defn{$\lambda$-weight space} of a $\fgl$-module $M$ in the usual manner. A polynomial representation of $\fgl$ is the direct sum of its weight spaces, and the weights that appear are non-negative (i.e., $\lambda_i \ge 0$ for all $i$). For more discussion and references on these notions, see \cite[\S 5.2]{expos}.

\subsection{Operads} \label{ss:operad}

A \defn{(linear) operad} $P$ is an $\FB$-module with extra structure:  
\begin{itemize}
\item for every finite set $S$, a composition map, denoted $\circ$:
  \[
    \bigoplus_{r \ge 0} P([r]) \otimes_{\bk[\fS_r]} \bigoplus_{S = a_1 \sqcup \cdots  \sqcup a_r} P(a_1) \otimes \dots \otimes P(a_r) \to P(S)
  \]
\item an identity element $\id_P \in P([1])$,
\end{itemize} 
satisfying an associativity and an identity axiom. Operads were originally introduced by May \cite{may}. We refer the reader to \cite[\S 5.2]{LodayValette} or  \cite[\S1]{Markl} for background on operads and algebras over them. 

Let $P$ be an operad. A \defn{$P$-algebra} consists of a vector space $A$, together with a map
\[
  p_* \colon A^{\otimes r} \to A,\quad   a_1 \otimes \dots \otimes a_r \mapsto p(a_1, \dots, a_r),
\]
for every $p \in P([r])$ satisfying equivariance, associativity, and identity axioms. We write $\Alg_P$ for the category of $P$-algebras.

\begin{remark}
 	Operads provide a framework for expressing certain algebraic structures on vector spaces.  (Namely algebraic structures on which consist of set of multilinear operations on a vector space $V$ that satisfy a set of universal identities).   One related framework is the notion of a \emph{polynomial identity algebra}, or {\it PI algebra}.  To any collection of  (noncommutative) polynomial identities $S$, one can associate a quotient of the associative operad $P:= \As/S$.  Then algebras satisfying the polynomial identities $S$ are equivalent to $P$-algebras.  
\end{remark}

Given a vector space $V$, there is a free $P$-algebra on $V$. This has underlying vector space $P\{V\}$, and the maps $p_*$ are defined using $\circ$. The identity element of $P([1])$ induces a natural map $V \to P\{V\}$. If $A$ is an arbitrary $P$-algebra the restriction along this map yields an isomorphism
\begin{displaymath}
\Hom_{\Alg_P}(P\{V\}, A) \to \Hom_{\Vec}(V, A).
\end{displaymath}
In other words, we may define a $P$-algebra homomorphism $P\{V\} \to A$ by specifying its values on a basis of $V$. 

\begin{remark}
A linear operad can also be called an operad in vector spaces. There is also a notion of operad in sets: such is an $\FB$-set (i.e., a functor from $\FB$ to the category of sets) with similar structure: just change direct sum to disjoint union and tensor product to cartesian product in the definition of $\circ$. Given an operad $P$ in sets, we can linearize it to get an operad $\bk P$ in vector spaces: $(\bk P)(S)$ is the vector space with basis $P(S)$. 
\end{remark}

\subsection{Derivations} \label{ss:deriv}

Let $P$ be a linear operad and let $A$ be a $P$-algebra. A \defn{derivation} of $A$ is a linear map  $\delta \colon A \to A$ that satisfies the Leibniz rule
\begin{displaymath}
\delta(p(a_1, \dots, a_r)) = \sum_{i = 1}^r p(a_1, \dots, \delta(a_i), \dots a_r),
\end{displaymath}
for every $p \in P([r])$ and $a_1, \dots, a_n \in A$.   The commutator of two derivations is a derivation,  so the collection of derivations forms a Lie algebra $\aut_P(A)$.  

Recall that $P\{V\}$ is the free $P$-algebra generated by $V$. There is a unique derivation $\delta \colon P\{V\} \to P\{V\}$ extending any linear map $V \to P\{V\}$. In other words there is a canonical identification
\[
  \aut_P (P\{V\}) = \Hom_{\Vec}(V, P\{V\})
\]
of vector spaces. Suppose that $V = \bk^n$ with standard basis $x_1, \dots, x_n$. Then given $p \in P\{\bk^n\}$ we write $p\del_{x_i}$ for the derivation corresponding to $p \otimes x_i^* \in \Hom(V,P\{V\})$. These elements span $\aut_P(P\{\bk^n\})$.

\begin{remark}
Derivations of a $P$-algebra $A$ are in canonical bijection with \emph{infinitesimal automorphisms of $A$}: automorphisms of $\bk[\epsilon]/(\epsilon^2) \otimes A$ (considered as a $P$ algebra in $\bk[\epsilon]/(\epsilon^2)$-modules) which induce the identity modulo $\epsilon$.
\end{remark}

\subsection{Wiring categories (PROPs)} \label{ss:wiring}

Let $P$ be an operad. The \defn{wiring category or (PROP)} associated to $P$ is a $\bk$-linear category $\cW_P$. Its objects are finite sets, and its morphism spaces are defined by
\[
\cW_P(S,T) := \bigoplus_{f \colon S \to T} \bigotimes_{i \in T}  P(f\inv(i)),
\]
where the sum is over all functions $f \colon S \to T$. Composition is defined on pure tensors by the formula
\begin{displaymath}
(f, (o_i))  \circ (g, (p_j)) = \left(f \circ g,  (o_i \circ (p_j)_{j \in f \inv(i)}) \right).
\end{displaymath}
Here $f \colon S \to T$ and $g \colon R \to S$ are functions and $o_i \in P(f^{-1}(i))$ for $i \in T$ and $p_j \in P(g^{-1}(j))$ for $j \in S$. 

\begin{remark}
The category $\cW_P$ is an example of a PROP, and the construction $P \mapsto \cW_P$ is the standard way of relating PROPs to operads  \cite[Example 60]{Markl}.
\end{remark}

\begin{remark}
There is also a notion of wiring category for an operad in sets; this is an ordinary category (not a $\bk$-linear category). The wiring category construction is compatible with linearization. Thus if $P$ is a linear operad that is obtained as the linearization of an operad in sets, its wiring category will be the linearization of an ordinary category. This is the case for many of the operads of interest.
\end{remark}

\subsection{Some examples} \label{ss:opex}

We now give some examples of the concepts introduced above.

\subsubsection{The trivial operad} \label{sss:trivial}

This operad has $P([1])=\bk$ and $P([n])=0$ for all $n \ne 1$. A $P$-algebra is simply a vector space, and the free $P$-algebra $P\{V\}$ is the vector space $V$. The Lie algebra $\aut_P(P\{V\})$ is the general linear Lie algebra $\fgl(V)$. The wiring category $\cW_P$ is $\bk[\FB]$, the linearization of $\FB$.

\subsubsection{The unital commutative operad} \label{sss:com}

This operad, denoted $\Com$, has $P(S)=\bk$ for all finite sets $S$. A $P$-algebra is a unital commutative associative $\bk$-algebra, and the free $P$-algebra $P\{V\}$ is the polynomial ring $\Sym(V)$ generated by $V$. The Lie algebra $\aut_P(P\{\bk^n\})$ is the Witt algebra $\fW_n$. The wiring category $\cW_P$ is $\bk[\Fin]$, where $\Fin$ is the category of all finite sets and maps between them.

\subsubsection{The non-unital commutative operad} \label{sss:comnu}

This operad, denoted $\Com^{\rn\ru}$, is defined just like $\Com$ except $P(\emptyset)=0$. A $P$-algebra is a commutative and associative, but not necessarily unital, $\bk$-algebra. The free $P$-algebra $P\{V\}$ is $\Sym_{\ge 1}(V)$, the positive degree part of $\Sym(V)$. The Lie algebra $\aut_P(P\{\bk^n\})$ is $\fW_n^{\pos}$. The wiring category $\cW_P$ is $\bk[\FS]$, where $\FS$ is the category of finite sets and surjections.

\subsubsection{The unital associative operad} \label{sss:assoc}

This operad, denoted $\As$, is the linearization of the operad in sets assigning to a finite set $S$ the collection of total orders on $S$. (The composition operation is given by composing totally ordered sets.) A $P$-algebra is a unital and associative $\bk$-algebra, and the free $P$-algebra $P\{V\}$ is $\rT(V)$, the tensor algebra on $V$. The category of \emph{non-commutative finite sets}, denoted $\NCFin$, is defined as follows. Its objects are finite sets. A morphism $S \to T$ is a function $f \colon S \to T$ equipped with a total ordering on each fiber $f^{-1}(t)$ for $t \in T$. We have $\cW_{\As}=\bk[\NCFin]$. The category $\NCFin$ was used by Ellenberg and Wiltshire-Gordon \cite{ellenberg-WG} to study configurations on manifolds which admit a non-vanishing vector field. There is also a non-unital version of this picture.

\subsubsection{The Lie operad} \label{sss:lie}

The operad ${\rm Lie}$ is the operad whose algebras are Lie algebras. The morphisms in $\cW$ are maps of finite sets whose fibers are decorated by Lie algebra operations. Representations of $\cW$ have been considered in \cite{Powell}. We will not discuss this example further in this paper, however.

\section{Equivalences of representation categories} \label{s:equiv}

In this section, we prove the equivalence of various representation categories associated to an operad. This includes Theorems~\ref{mainthm3} and~\ref{mainthm4}.

\subsection{Statement of results} \label{ss:setup}

We fix the following data for this section:
\begin{description}[align=right,labelwidth=2.5cm,leftmargin=!]
\item[ $P$ ] An operad
\item[ $\bW_n$ ] The vector space $\bk^n$, which carries the standard representation of $\fgl_n$
\item[ $\bV_n$ ] The free $P$-algebra $P\{\bW_n\}$
\item[ $\fh_n$ ] The Lie algebra $\aut_P(\bV_n)$ of derivations of the $P$-algebra $\bV_n$
\item[ $\fM_n$ ] The monoid $\End(\bV_n)$ of all $P$-algebra endomorphisms of $\bV_n$
\item[ $\cC$ ] The full subcategory of $\Alg_P$ spanned by the $\bV_n$'s
\item[ $\cW$ ] The wiring category $\cW_P$ of $P$
\end{description}
We also let $\bW=\bigcup_{n \ge 1} \bW_n$, and similarly define $\bV$, $\fh$, and $\fM$. We note that $\fgl_n$ naturally acts on $\bV_n$ by algebra derivations, which gives an inclusion $\fgl_n \subset \fh_n$; similarly, $\fgl \subset \fh$. We also note that there are natural inclusions $\bW_n \subset \bV_n$ and $\bW \subset \bV$ coming from the identity element of $P([1])$.

The space $\bV_n$ carries tautological actions of $\fh_n$ and $\fM_n$. We define a representation of $\fh_n$ or $\fM_n$ to be \defn{polynomial} if it occurs as a subquotient of a sum of tensor powers of $\bV_n$. We make similar definitions for $\fh$ and $\fM$. Similarly, there is a tautological $\cC$-module given by $\bV_n \mapsto \bV_n$, and we say that a $\cC$-module is \defn{polynomial} if it occurs as a subquotient of a sum of tensor powers of this module (where tensor products are defined pointwise). In all cases, we write $\Rep^{\pol}(-)$ for the category of polynomial representations.

The following theorem is our main result on categorical equivalences.  It is an operadic generalization of Schur-Weyl duality.   It implies Theorem~\ref{mainthm4}, and, when $P=\Com$, Theorem~\ref{mainthm3} (see \S \ref{sss:com}). The proof will take the remainder of this section.

\begin{theorem} \label{thm:equiv}
The following categories are naturally equivalent:
\begin{displaymath}
\Rep^{\pol}(\fh), \qquad
\Rep^{\pol}(\fM), \qquad
\Rep^{\pol}(\cC), \qquad
\Rep(\cW^{\op}).
\end{displaymath}
\end{theorem}

We can now appeal to known noetherianity results for combinatorial categories to obtain noetherianity results for the Lie algebra $\fh$.

\begin{corollary} \label{cor:comm-noeth}
Suppose $P$ is either $\Com$ or $\Com^{\rm nu}$. Then  $\Rep^{\pol}(\fh)$ is locally noetherian.
\end{corollary}

\begin{proof}
If $P=\Com^{\rm nu}$ then $\cW = \bk[\FS]$ (\S \ref{sss:comnu}), and the category of $\cW^{\op}$-modules is locally noetherian by \cite[Corollary 8.1.3]{catgb}. If $P=\Com$ then $\cW=\bk[\Fin]$ (\S \ref{sss:com}), and the category of $\cW^{\op}$-modules is locally noetherian by \cite[Corollary 8.4.2]{catgb}.
\end{proof}

\subsection{Action of wiring categories}

Let $R$ be a $P$-algebra. Suppose that $\phi \colon [n] \to [m]$ is a morphism in $\cW$, corresponding to a function $f \colon [n] \to [m]$ and operations $p_i \in P(f^{-1}(i))$ for each $i \in [m]$ (i.e., $\phi$ lives in the $f$ summand of the $\Hom$ space, and is a pure tensor). We define a map $\phi_* \colon R^{\otimes n} \to R^{\otimes m}$ by
\begin{displaymath}
x_1 \otimes \cdots \otimes x_n \mapsto p_1(x_{j})_{j \in f \inv(1)} \otimes \dots \otimes p_m(x_{j})_{j \in f \inv(m)}.
\end{displaymath}
We extend this definition linearly to all morphisms $\phi$ in $\cW$. One readily verifies the following properties of this construction:
\begin{enumerate}
\item $\phi_*$ is a map of $P$-algebras, using the natural $P$-algebra structure on tensor powers.
\item It is functorial, that is, if $\psi \colon [m] \to [\ell]$ is a second morphism in $\cW$ then $(\psi \circ \phi)_*=\psi_* \circ \phi_*$.
\item The construction $\phi_*$ is natural in $R$.
\end{enumerate}
In particular, applying this in the case where $R=\bV$, we obtain a map $\phi_* \colon \bV^{\otimes n} \to \bV^{\otimes m}$ which commutes with the action of $\fgl$ by naturality. We can restrict the domain of this map to $\bW^{\otimes n} \subset \bV^{\otimes n}$ to obtain a map of $\fgl$-representations. We thus have a natural map
\begin{equation} \label{eq:Wpdef}
\Hom_{\cW}([n], [m]) \to \Hom_{\fgl}(\bW^{\otimes n}, \bV^{\otimes m})
\end{equation}
We will require the following result about this construction:

\begin{proposition}\label{prop:Wpdef}
The map \eqref{eq:Wpdef} is an isomorphism.
\end{proposition}
      
\begin{proof}
If $M$ is an $\FB$-module then we have a polynomial representation $M\{\bW\}$ of $\fgl$, and there is a natural identification
\begin{displaymath}
\Hom_{\fgl}(\bW^{\otimes n}, M\{\bW\}) = M([n]).
\end{displaymath}
We apply this to the $\FB$-module $M=P^{\otimes m}$, where the tensor product on the right side is of $\FB$-modules (see \S \ref{ss:poly}). We have $M\{\bW\}=\bV^{\otimes m}$. We thus find
\begin{displaymath}
\Hom_{\fgl}(\bW^{\otimes n}, \bV^{\otimes m}) = (P^{\otimes m})([n]) = \bigoplus_{[n] = a_1 \sqcup \dots \sqcup a_m} \bigotimes_{i = 1}^m  P(a_i),
\end{displaymath}
where in the second step we used the definition of tensor product of $\FB$-modules. This last vector space is exactly $\Hom_\cW([n],[m])$.  Tracing through the inverse isomorphism, we see that it is as described in the proposition.  
\end{proof}

\subsection{Comparing $\fh$ and $\cW$}

We now establish the equivalence of the categories $\Rep^{\pol}(\fh)$ and $\Rep(\cW^{\op})$. We first need to establish some basic results about these categories.

\begin{proposition} \label{prop:generators}
The $\fh$-module $\bV^{\otimes n}$ is generated by the subspace $\bW^{\otimes n}$.
\end{proposition}

\begin{proof}
Let $m \in \bV^{\otimes n}$ be an element of the form $m = p_1\otimes \dots \otimes p_n$ where $p_i \in \bV$ is a weight vector. It suffices to show that $m$ is in the submodule generated by $\bW^{\otimes n}$. Let $\{x_i\}_{i \ge 1}$ be the standard basis of $\bW$. Let $a(1), \ldots, a(n)$ be distinct indices not appearing in the supports of the weights of the $p_i$'s. Let $x \in \bW^{\otimes n}$ be the pure tensor $x_{a(1)} \otimes \dots \otimes x_{a(n)}$, and let $Y=(p_1 \del_{x_{a(1)}}) \cdots (p_n \del_{x_{a(n)}})$, which is an element of the universal enveloping algebra of $\fh$. We have $\partial_{a(i)}(p_j)=0$ for all $i$ and $j$ and $\partial_{a(i)}(x_{a(j)})=\delta_{i,j}$. It follows that $Yx=m$, which completes the proof.
\end{proof}

Given a $\cW$-morphism $\phi \colon [n] \to [m]$, we have seen that there is an associated map $\phi_* \colon \bV^{\otimes n} \to \bV^{\otimes m}$ that is $\fgl$-equivariant. In fact, since $\phi$ is natural in $\bV$, it commutes with $\fM$, the endomorphism monoid of $\bV$. It is not difficult to see that it also commutes with $\fh$, as elements of $\fh$ can be seen as infinitesimal automorphisms of $\bV$. We thus have maps
\begin{equation} \label{eq:homg}
\Hom_{\cW}([n], [m]) \to \Hom_{\fh}(\bV^{\otimes n}, \bV^{\otimes m}) \to \Hom_{\fgl}(\bW^{\otimes n}, \bV^{\otimes m}),
\end{equation}
where the second map is restriction to the $\fgl$-subrepresentation $\bW^{\otimes n} \subset \bV^{\otimes n}$.

\begin{proposition} \label{prop:homg}
Both maps in \eqref{eq:homg} are isomorphisms.
\end{proposition}

\begin{proof}
The second map is injective since $\bW^{\otimes n}$ generates $\bV^{\otimes n}$ as a $\fh$-module (Proposition~\ref{prop:generators}). Since the composite map is an isomorphism (Proposition~\ref{prop:Wpdef}), the result follows.
\end{proof}

The following proposition is the main result about $\fh$ that we will need.

\begin{proposition} \label{prop:gproj}
For any polynomial $\fh$-representation $M$, the restriction map
\begin{displaymath}
\Hom_{\fh}(\bV^{\otimes n}, M) \to \Hom_{\fgl}(\bW^{\otimes n}, M)
\end{displaymath}
is an isomorphism. In particular, $\bV^{\otimes n}$ is a projective object of $\Rep^{\pol}(\fh)$.
\end{proposition}

\begin{proof}
The map is injective since $\bW^{\otimes n}$ generates $\bV^{\otimes n}$ (Proposition~\ref{prop:generators}). We must prove surjectivity. Thus suppose that $\phi \colon \bW^{\otimes n} \to M$ is a given $\fgl$-equivariant map. Write $M=A/B$ where $B \subset A$ are submodules of $X=\bigoplus_{i \in I} \bV^{\otimes m_i}$; this is possible by the definition of polynomial representation. We thus have a surjection $\pi \colon A \to M$. Let $\psi \colon \bW^{\otimes n} \to A$ be a lift of $\phi$ (i.e., $\phi=\pi \circ \psi$), which exists since polynomial representations of $\fgl$ are semisimple. By Proposition~\ref{prop:homg}, $\psi$ is the restriction to $\bW^{\otimes n}$ of a unique map of $\fh$-modules $\tilde{\psi} \colon \bV^{\otimes n} \to X$. Since $\tilde{\psi}$ maps $\bW^{\otimes n}$ into $A$ and $\bW^{\otimes n}$ generates $\bV^{\otimes n}$, it follows that $\tilde{\psi}$ maps into $A$. Thus $\tilde{\phi} = \pi \circ \tilde{\psi}$ is a map $\bV^{\otimes n} \to M$ of $\fh$-modules that restricts to $\phi$ on $\bW^{\otimes n}$. This proves surjectivity. Since $\Hom_{\fgl}(\bW^{\otimes n}, -)$ is an exact functor, it follows that $\bV^{\otimes n}$ is projective.
\end{proof}

We can now prove part of the main theorem:

\begin{proposition}
The categories $\Rep(\cW^{\op})$ and $\Rep^{\pol}(\fh)$ are naturally equivalent.
\end{proposition}

\begin{proof}
Let $P_n = \Hom_{\cW^{\op}}([n], -)$ be the principal projective $\cW^{\op}$-module. The $P_n$ are compact projective generators. Let $\cP$ be the full subcategory of $\Rep(\cW^{\op})$ that they span. The functor $\cW \to \cP$ given by $[n] \mapsto P_n$ is easily seen to be an equivalence.

Let $Q_n=\bV^{\otimes n}$. It follows from Proposition~\ref{prop:gproj} that the $Q_n$'s are compact projective generators of $\Rep^{\pol}(\fh)$. Let $\cQ$ be the full subcategory that they span. Proposition~\ref{prop:homg} shows that the functor $\cW \to \cQ$ given by $[n] \mapsto Q_n$ is an equivalence.

Combining the above two paragraphs, we see that $\cP$ and $\cQ$ are equivalent. By a standard result from Morita theory, the equivalence $\cP \to \cQ$ uniquely extends to an equivalence $\Rep(\cW^{\op}) \to \Rep^{\pol}(\fh)$. (Briefly, given a $\cW^{\op}$-module, choose a presentation of it by objects in $\cP$, and apply the functor $\cP \to \cQ$ to obtain a presentation for the corresponding object of $\Rep^{\pol}(\fh)$.)
\end{proof}

\subsection{Representations of $\fM$}

We now examine polynomial representations of $\fM$. We first show that we can essentially differentiate such a representation to get a polynomial representation of $\fh$. Let $\{x_i\}_{i \ge 1}$ be a basis for $\bW$. Suppose $f \in \bV$ is homogeneous, $i \ge 1$, and $t \in \bk$. We define an element $\gamma_{f,i,t} \in \fM$ as follows:
\begin{displaymath}
\gamma_{f,i,t}(x_j)= \begin{cases} x_i+tf & \text{if $j=i$}\\ x_j & \text{if $j \ne i$} \end{cases}.
\end{displaymath}
Here we are defining $\gamma_{f,i,t}$ on $\bW$, which is sufficient since $\bV$ is a free $P$-algebra. We regard this as a 1-parameter family in $\fM$ (as $t$ varies).

\begin{proposition}\label{derivativerep}
Let $M$ be a polynomial $\fM$-module and let $m \in M$. Let $f$, $i$, and $t$ be as above. Then there exist unique elements $m_0, \ldots, m_n \in M$ such that
\addtocounter{equation}{-1}
\begin{subequations}
\begin{equation}
  \gamma_{f,i,t}(m) = \sum_{j=0}^n m_j t^j \label{derivative}
\end{equation}
\end{subequations}
for all $t \in \bk$. Moreover, $M$ carries a unique representation of $\fh$ satisfying $f \del_{x_i} \cdot m = m_1$, and this representation is polynomial.
\end{proposition}

\begin{proof}
Because $\bk$ is infinite, the $m_i$ are unique if they exist,  so we show existence.  

If $M \subseteq N$ and there are $m_i \in N$  satisfying \eqref{derivative} then $m_i \in M$,  because $\bk$ is infinite and $\gamma_{f,i,t}(m) \in M$.

Letting $M = A/B$  for $B \subseteq A \subseteq \bigoplus_{i}  \bV^{\otimes d_i}$,  we may reduce first to the case $M = A$,  and then $M = \bigoplus_{i} \bV^{\otimes d_i}$ by the above.  By linearity of \eqref{derivative}, we may reduce to the case where $m \in \bV^{\otimes d}$ is a simple tensor of generators of $\bV$.  Multiplying terms, we reduce to the case where $m = p(x_{i_1}, \dots, x_{i_r}) \in \bV$  for $p \in P([r])$,  for which it follows by multilinearity.  

To check that $f \del_{x_i} \cdot m = m_1$ defines a $\fh$ representation, we may again reduce to $\bV^{\otimes d}$.  In this case, the proposed action of $f \del_{x_i}$ agrees with the standard one,  hence defines a representation.   Since $M$ is a subquotient of a sum of powers of $\bV$, the representation is polynomial.
\end{proof}

By the above proposition, we have a functor
\begin{equation} \label{eq:fM-fg}
\Rep^{\pol}(\fM) \to \Rep^{\pol}(\fh).
\end{equation}
We can now take the next step in the proof of the theorem:

\begin{proposition} \label{prop:fM-fg}
The functor \eqref{eq:fM-fg} is an equivalence.
\end{proposition}

\begin{proof}
Let $M$ be a polynomial $\fM$-module. From the construction of the functor, any $\fM$-submodule of $M$ is also an $\fh$-submodule. Since $\bW^{\otimes n}$ generates $\bV^{\otimes n}$ as an $\fh$-module (Proposition~\ref{prop:generators}), it also generates as an $\fM$-module. Given a map $\phi \colon [n] \to [m]$ in $\cW$, the induced map $\phi_* \colon \bV^{\otimes n} \to \bV^{\otimes m}$ commutes with $\fM$. We thus have maps
\begin{displaymath}
\Hom_{\cW}([n], [m]) \to \Hom_{\fM}(\bV^{\otimes n}, \bV^{\otimes m}) \to \Hom_{\fh}(\bV^{\otimes n}, \bV^{\otimes m}).
\end{displaymath}
We have already seen that the composition is an isomorphism (Proposition~\ref{prop:homg}). Since the second map is obviously injective, it is thus also an isomorphism. The reasoning used in Proposition~\ref{prop:gproj} now applies equally well to $\fM$, and we find that $\bV^{\otimes n}$ is projective in $\Rep^{\pol}(\fM)$. The functor in question induces an equivalence between the full subcategories spanned by the modules $\bV^{\otimes n}$. Since these are compact projective generators in each category, it follows that the functor is an equivalence.
\end{proof}

\subsection{Representations of $\cC$}

The following proposition completes the proof of Theorem~\ref{thm:equiv}.

\begin{proposition} \label{prop:C-equiv}
We have a natural equivalence $\Rep^{\pol}(\cC)=\Rep^{\pol}(\fM)$.
\end{proposition}

\begin{proof}
Let $T_n$ be the $\cC$-module given by $T_n(\bV_r)=\bV_r^{\otimes d}$. Thus a $\cC$-module is polynomial if and only if it is a subquotient of a sum of $T_n$'s. For a $\cC$-module $M$, define $\Phi(M)=\varinjlim M(\bV_r)$. Here the direct limit is taken with respect to the standard inclusions $\bV_r \to \bV_{r+1}$ (which, in turn, are induced by the standard inclusions $\bW_r \to \bW_{r+1}$). The functor $\Phi$ is exact, cocontinuous, and satisfies $\Phi(T_n)=\bV^{\otimes n}$. It follows that $\Phi$ induces a functor
\begin{displaymath}
\Phi \colon \Rep^{\pol}(\cC) \to \Rep^{\pol}(\fM).
\end{displaymath}
We show that this is an equivalence.

We first claim that $\Phi$ is faithful. Let $f \colon M \to N$ be a map of $\cC$-modules such that $\Phi(f)=0$. The inclusion $\bV_r \to \bV_{r+1}$ is a split monomorphism in $\cC$ since the standard projection $\bW_{r+1} \to \bW_r$ induces a left inverse $\bV_{r+1} \to \bV_r$. It follows that the map $M(\bV_r) \to M(\bV_{r+1})$ is injective, and similarly for $N$. This implies that $f \colon M(\bV_r) \to N(\bV_r)$ vanishes for all $r$, and so $f=0$. Thus $\Phi$ is faithful.

Suppose $\phi \colon [n] \to [m]$ is a morphism in $\cW$. Then $\phi_* \colon T_n \to T_m$ is a map of $\cC$-modules; applying $\Phi$ to this yields the map $\phi_* \colon \bV^{\otimes n} \to \bV^{\otimes m}$. We thus have maps
\begin{displaymath}
\Hom_{\cW}([n], [m]) \to \Hom_{\cC}(T_n, T_m) \to \Hom_{\fM}(\bV^{\otimes n}, \bV^{\otimes m}),
\end{displaymath}
where the composition is the natural map, which we have already seen is an isomorphism (in the proof of Proposition~\ref{prop:fM-fg}). Since $\Phi$ is faithful, the second map is injective. It thus follows that both maps above are isomorphisms.

We can now apply the same argument from previous cases to see that the $T_n$'s are projective generators of $\Rep^{\pol}(\cC)$. Since $\Phi$ induces an equivalence of categories of compact projective generators, it follows that $\Phi$ is an equivalence, which completes the proof.
\end{proof}

\begin{remark}
The above theorem can be extended to $\Alg_P$, and the subcategory $\Alg_P^{\rm fp}$ of finitely presented algebras as follows. We say that a $\Alg_P$-module $F$ is \defn{right exact} if it takes coequalizers to right-exact sequences, and that it is \defn{finitary} if it preserves filtered colimits. The following categories are equivalent:
\begin{enumerate}
\item the category of right exact, finitary functors $\Alg_P \to \Vec$.
\item the category of right exact functors  $\Alg_P^{\rm fp} \to \Vec$.
\item the category of all functors $\cC \to \Vec$.
\end{enumerate}
This is a standard result in categorical universal algebra (see \cite[Theorem 7.3]{ARV}). These equivalences induces equivalences of the categories of polynomial representations.
\end{remark}

\subsection{Some additional results}

In the course of proving Theorem~\ref{thm:equiv}, we showed that $\bV^{\otimes n}$ is a projective object of $\Rep^{\pol}(\fh)$ (Proposition~\ref{prop:gproj}). In fact, we can classify the projectives and the simples completely (at least in some cases), as the next two results show. We let $\bS_{\lambda}$ denote the Schur functor corresponding to the partition $\lambda$.

\begin{proposition} \label{prop:gproj2}
For a polynomial $\fh$-module $M$ and a partition $\lambda$, the restriction map
\begin{displaymath}
\Hom_{\fh}(\bS_{\lambda}(\bV), M) \to \Hom_{\fgl}(\bS_{\lambda}(\bW), M)
\end{displaymath}
is an isomorphism. Thus $\bS_{\lambda}(\bV)$ is a projective object of $\Rep^{\pol}(\fh)$.
\end{proposition}

\begin{proof}
This follows from Proposition~\ref{prop:gproj} by decomposing under the action of the symmetric group $\fS_n$.
\end{proof}

\begin{proposition} \label{prop:class-proj}
Suppose that $P([0])=0$ and $P([1])=\bk$. Then:
\begin{enumerate}
\item The simple objects of $\Rep^{\pol}(\fh)$ are the representations $\bS_{\lambda}(\bW)$.
\item The indecomposable projectives of $\Rep^{\pol}(\fh)$ are the representations $\bS_{\lambda}(\bV)$.
\end{enumerate}
\end{proposition}

\begin{proof}
We have an identification $\fh_n = \Hom(\bW_n,P\{\bW_n\})$ (see \S \ref{ss:deriv}), and hence by our assumptions on $P$, the algebra $\fh_n$ is concentrated in non-negative degrees and its degree $0$ component is $\fgl_n$. It follows that $\fh$ is concentrated in non-negative degrees and its degree $0$ component is $\fgl$. Thus the elements of $\fh$ of strictly positive degree form an ideal with quotient $\fgl$. 

Let $M$ be a polynomial $\fh$ representation, and for $d \in \bN$, let $M_d \subseteq M$ be the minimal degree $\fgl$-subrepresentation. Since $P([0]) = 0$, the subspace of elements of degree $> d$ is a subrepresentation of $M$, and hence $M$ surjects onto $M_d$. If $M$ is simple, it follows that $M = M_d$ and the action of $\fh$ on $M_d$ factors through $\fgl$. Hence by the classification of simple $\fgl$ modules, $M = \bS_{\lambda}(\bW)$ for some partition $\lambda$. (Here $\fh$ acts on $M$ through the quotient homomorphism).

From our hypotheses, it follows that the degree $|\lambda|$ component of $\bS_{\lambda}(\bV)$ is $\bS_{\lambda}(\bW)$. Thus Proposition~\ref{prop:gproj2} yields $\End_{\fh}(\bS_{\lambda}(\bV)) = \bk$, and so $\bS_{\lambda}(\bV)$ is indecomposable with unique simple quotient $\bS_{\lambda}(\bW)$. Since every simple representation is of the form $\bS_{\lambda}(\bW)$, every indecomposable projective is isomorphic to $\bS_{\lambda}(\bV)$ for some $\lambda$.
\end{proof}

\begin{remark}
The above hypotheses are met if $P=\Com^{\rn\ru}$, but not if $P=\Com$.
\end{remark}

\subsection{An example} \label{ss:kahler}

Let $P=\Com$ be the unital commutative operad. Then $\cC$ is the category of (commutative) polynomial algebras in finitely many variables and $\cW$ is the (linearization of) $\Fin$; see \S \ref{sss:com}.

Consider the $\cC$-module $M$ given by $M(R)=\Omega^1_{R/\bk}$, which takes a $\bk$-algebra to its space of K\"ahler forms. The module $M$ is polynomial, since the definition of K\"ahler forms using the ideal of the diagonal realizes $M$ as a subquotient of the module $T_2$ defined in the proof of Proposition~\ref{prop:C-equiv}.

Let $N$ be the $\Fin^{\op}$-module corresponding to $M$ under the equivalence in Theorem~\ref{thm:equiv}. We claim that $N(S)=\bk[S]^*$. To see this, notice that the $1^n$ weight space of $\Omega^1_{\bk[x_1, \dots, x_n]}$ is generated by $\{\omega_j \}_{j = 1}^n$,  where $\omega_j := (\prod_{i \in [n] \setminus \{j\}} x_i )dx_j$.  A function $f\colon [m] \to [n]$ acts by
\begin{displaymath}
f^* \omega_j = \left(\prod_{a \in f\inv([m] \setminus \{j\})} x_a \right) d(\prod_{b \in f\inv(j)} x_b) = \sum_{b \in f\inv(j)} \omega_b.
\end{displaymath}
This matches with the action of $f^* \colon (\bk^n)^* \to (\bk^m)^*$ in the standard basis.

A similar analysis shows that the $\cC$-module $R \mapsto \Omega^d_{R/\bk}$ corresponds to the dual of $X \mapsto \lw^d(\bk[X])$. The de Rham complex
\begin{displaymath}
R \to \Omega^1_R \to \Omega^2_R \to \cdots
\end{displaymath}
corresponds to the dual of the Koszul complex constructed by Wiltshire--Gordon \cite[Definition 4.22]{wiltshire-gordon}

\section{The finite dimensional setting} \label{s:findim}

In the previous section, we studied representations of the Lie algebra $\fh$, which is the direct limit of the Lie algebras $\fh_n$. We now study the $\fh_n$'s. A key tool is the \emph{specialization functor} which takes representations of $\fh$ to representations of $\fh_n$. We maintain the set-up from \S \ref{ss:setup} throughout this section.

\subsection{The specialization functor}

Let $\fgl_n \times \fgl_{\infty-n}$ be the subalgebra of $\fgl$ consisting of block diagonal matrices with one block of size $n \times n$ and the other block using the remaining rows and columns. For a polynomial $\fh$-module $M$, we define
\begin{displaymath}
\Gamma_n(M) = M^{\fgl_{\infty-n}},
\end{displaymath}
where the superscript denotes invariants. Since $\fh_n$ commutes with $\fgl_{\infty-n}$ inside of $\fh$, there is an action of $\fh_n$ on $\Gamma_n(M)$. The following proposition establishes the basic properties of this construction.

\begin{proposition}
We have the following:
\begin{enumerate}
\item The construction $\Gamma_n$ is exact, co-continuous, and symmetric monoidal.
\item We have $\Gamma_n(\bV)=\bV_n$.
\item If $M$ is a polynomial $\fh$-module then $\Gamma_n(M)$ is a polynomial $\fh_n$-module.
\end{enumerate}
\end{proposition}

\begin{proof}
(a) This is simply a fact about $\fgl_{\infty-n}$-invariants of polynomial $\fgl$-representations: note that the construction $\Gamma_n(M)$ depends only on the $\fgl$-module structure on $M$. (From the point of view of polynomial functors, taking invariants under $\fgl_{\infty-n}$ is just evaluating on $\bW_n$.)

(b) This is clear from direct calculation.

(c) From (a) and (b), we see that $\Gamma_n(\bV^{\otimes d})=\bV_n^{\otimes d}$. Since $\Gamma_n$ is exact and co-continuous, the statement follows.
\end{proof}

We thus have a functor
\begin{displaymath}
\Gamma_n \colon \Rep^{\pol}(\fh) \to \Rep^{\pol}(\fh_n).
\end{displaymath}
We call this the \defn{specialization functor}.

\begin{remark}
Identifying $\Rep^{\pol}(\fh)$ with $\Rep^{\pol}(\cC)$, the specialization functor amounts to evaluating on $\bV_n$.
\end{remark}

\subsection{Specializations of projective modules}

Since $\bV$ specializes to $\bV_n$, it follows that $\bS_{\lambda}(\bV)$ specializes to $\bS_{\lambda}(\bV_n)$. The following proposition establishes some basic results about these modules, at least for certain $\lambda$.

\begin{proposition} \label{prop:gendeg2}
Let $\lambda$ be a partition with $\le n$ rows.
\begin{enumerate}
\item $\bS_{\lambda}(\bV_n)$ is generated as an $\fh_n$-module by $\bS_{\lambda}(\bW_n)$.
\item For a polynomial $\fh_n$-module $M$, the restriction map
\begin{displaymath}
\Hom_{\fh_n}(\bS_{\lambda}(\bV_n), M) \to \Hom_{\fgl_n}(\bS_{\lambda}(\bW_n), M)
\end{displaymath}
is an isomorphism.
\item For a polynomial $\fh$-module $N$, the natural map
\begin{displaymath}
\Gamma_n \colon \Hom_{\fh}(\bS_{\lambda}(\bV), N) \to \Hom_{\fh_n}(\bS_{\lambda}(\bV_n), \Gamma_n(N))
\end{displaymath}
is an isomorphism.
\item The object $\bS_{\lambda}(\bV_n)$ is projective. If $P([0])=0$ and $P([1])=\bk$ then it is indecomposable, and all indecomposable projectives have this form.
\item Every polynomial $\fh_n$-module is a quotient of a sum of modules of the form $\bS_{\nu}(\bV_n)$ with $\ell(\nu) \le n$.
\end{enumerate}
\end{proposition}

In proving the above proposition, we will obtain the following additional result.

\begin{proposition} \label{prop:ess-surj}
The specialization functor $\Gamma_n$ is essentially surjective.
\end{proposition}

We begin with some lemmas.

\begin{lemma} \label{lem:gendeg2-1}
If $\ell(\lambda) \le n$ then $\bS_{\lambda}(\bV_n)$ is generated by $\bS_{\lambda}(\bW_n)$ as an $\fh_n$-module.
\end{lemma}

\begin{proof}
Let $x_1, x_2, \ldots$ be the standard basis for $\bW$. For a weight vector $f$ of $\bV$, let $X_{f,i}=f \partial_{x_i}$, regarded as an element of $\fh$. These elements span $\fh$. By (the easy part of) Poincar\'e--Birkhoff--Witt, the universal enveloping algebra $U(\fh)$ of $\fh$ is spanned by elements of the form $X_{f_1,i_1} \cdots X_{f_r,i_r}$ with $i_1 \le i_2 \le \cdots \le i_r$. We call these \defn{admissible elements}.

Let $m_0$ be a non-zero weight vector in $\bS_{\lambda}(\bW_n)$ and let $m$ be an arbitrary weight vector in $\bS_{\lambda}(\bV_n)$. Since $\bS_{\lambda}(\bW)$ is an irreducible representation of $\fgl$, it is generated by $m_0$. Since $\bS_{\lambda}(\bV)$ is generated by $\bS_{\lambda}(\bW)$ as an $\fh$-module (Proposition~\ref{prop:generators}), it is thus generated by the single element $m_0$. We can therefore write $m=\sum_{i=1}^s Y_i m_0$ where $Y_i \in U(\fh)$ is an admissible element. Note that since each $Y_i m_0$ is a weight vector, we can assume that it has the same weight as $m$, since all terms for which this is not true must cancel. Of course, we can also assume that $Y_i m_0$ is non-zero for each $i$.

Let us examine the element $Ym_0$ where $Y=X_{f_1,i_1} \cdots X_{f_r,i_r}$ is admissible. Since the support of the weight of $m_0$ is contained in $\{1,\ldots,n\}$, it follows that $\partial_{x_i} m_0=0$ if $i>n$. Thus if $Ym_0$ is non-zero then $i_r \le n$; by admissibility, it follows that each index $i_j$ is $\le n$. If additionally the support of the weight of $Ym_0$ is contained in $\{1, \ldots, n\}$ then the same must be true for the weight of each $f_i$, since the partial derivatives appearing in $Y$ cannot decrease the weight outside of this set of indices. It follows that $f_i$ belongs to $\bV_n$, and so $Y$ actually belongs to $U(\fh_n)$.

Returning to our expression for $m$, we thus see that each $Y_i$ belongs to $U(\fh_n)$. This shows that the $\fh_n$-submodule of $\bS_{\lambda}(\bV_n)$ generated by $m_0$ contains $m$. Since $m$ is arbitrary, the result follows.
\end{proof}

\begin{lemma} \label{lem:gendeg2-2}
Suppose $\ell(\lambda) \le n$ and $N$ is a polynomial $\fh$-module. Then the natural maps
\begin{displaymath}
\Hom_{\fh}(\bS_{\lambda}(\bV), N) \to \Hom_{\fh_n}(\bS_{\lambda}(\bV_n), \Gamma_n(N)) \to \Hom_{\fgl_n}(\bS_{\lambda}(\bW_n), \Gamma_n(N))
\end{displaymath}
are isomorphisms.
\end{lemma}

\begin{proof}
Consider the commutative diagram
\begin{displaymath}
\xymatrix{
\Hom_{\fh_n}(\bS_{\lambda}(\bV_n), \Gamma_n(N)) \ar[r] & \Hom_{\fgl_n}(\bS_{\lambda}(\bW_n), \Gamma_n(N)) \\
\Hom_{\fh}(\bS_{\lambda}(\bV), N) \ar[r] \ar[u] & \Hom_{\fgl}(\bS_{\lambda}(\bW), N) \ar[u] }
\end{displaymath}
The top map is injective by Lemma~\ref{lem:gendeg2-1}. We have already shown that the bottom map is an isomorphism (Proposition~\ref{prop:gproj2}). The right map is an isomorphism using standard facts about polynomial representations of $\fgl$; the key point here is that because $\ell(\lambda) \le n$ the representation $\bS_{\lambda}(\bW_n)$ is non-zero and irreducible. Thus the left and top maps are isomorphisms.
\end{proof}

\begin{proof}[Proof of Proposition~\ref{prop:gendeg2}]
Lemma~\ref{lem:gendeg2-1} proves statement (a). Lemma~\ref{lem:gendeg2-2} proves statement (c), and statement (b) for modules $M$ that belong to the essential image of $\Gamma_n$.

We now prove (e). Let $M$ be a polynomial representation of $\fh_n$. Thus $M=A/B$ where $B \subset A$ are $\fh_n$-submodules of $X=\bigoplus_{i \in I} \bS_{\mu_i}(\bV_n)$, where the $\mu_i$ are arbitrary partitions. Let $A=\bigoplus_{j \in J} \bS_{\nu_j}(\bW_n)$ be the decomposition of $A$ as a $\fgl_n$-module, where the $\nu_j$ are partitions with $\le n$ rows. Since $X$ is in the essential image of $\Gamma_n$, we see that (b) holds for it. Thus the inclusion $\bS_{\nu_j}(\bW_n) \to X$ is the restriction of a unique map $\bS_{\nu_j}(\bV_n) \to X$ of $\fh_n$-modules. Since $\bS_{\nu_j}(\bW_n)$ maps into $A$ and generates $\bS_{\nu_j}(\bV_n)$, it follows that all of $\bS_{\nu_j}(\bV_n)$ maps into $A$. We thus have a map of $\fh_n$-modules $\bigoplus_{j \in J} \bS_{\nu_j}(\bV_n) \to A$, which is clearly surjective. Since $M$ is a quotient of $A$, we also have such a surjection with $M$ in place of $A$. This proves (e).

We now prove that $\Gamma_n$ is essentially surjective. Thus let $M$ be a polynomial $\fh_n$-module. Applying (e) twice, we can find a presentation
\begin{displaymath}
P_1 \stackrel{f}{\to} P_0 \to M \to 0
\end{displaymath}
where $P_0$ and $P_1$ are sums of $\bS_{\nu}(\bV_n)$ with $\ell(\nu) \le n$. Let $Q_0$ and $Q_1$ be the corresponding sums of $\bS_{\nu}(\bV)$'s. Thus $P_i=\Gamma_n(Q_i)$. By (c), the map
\begin{displaymath}
\Hom_{\fh}(Q_1, Q_0) \to \Hom_{\fh_n}(P_1, P_0)
\end{displaymath}
is an isomorphism. Thus $f$ is the specialization of a unique map $g \colon Q_1 \to Q_0$. Let $N=\coker(g)$. Since $\Gamma_n$ is exact, we have $M=\Gamma_n(N)$, as required. Since $\Gamma_n$ is essentially surjective, this proves (b) in general.

Statement (d) follows from an argument as in the proof of Proposition~\ref{prop:class-proj}.
\end{proof}

\begin{example}
Proposition~\ref{prop:gendeg2} shows that $\bS_{\lambda}(\bV_n)$ is projective and generated in degree $\vert \lambda \vert$ if $\ell(\lambda) \le n$. These statements fail if $\ell(\lambda)>n$. To see this, we take $P=\Com^{\rn\ru}$ and $n=1$ in what follows. Note that $\fh_1=\fW^{\pos}_1$ in this case and $\bV_1$ is the positive degree part of the polynomial ring $\bk[t]$. (See \S \ref{sss:comnu}.)

The module $M=\lw^2(\bV_1)$ is nonzero since $\bV_1$ is infinite dimensional, but vanishes in degree~2 since $\bW_1$ is one-dimensional. Thus $M$ is \emph{not} generated in degree~2. Consequently, $\bV_1^{\otimes 2}$ is also not generated in degree~2.

We now show that $M$ is not projective. Indeed, suppose it were. By Proposition~\ref{prop:gendeg2}, the indecomposable projective $\fh_1$-modules have the form $\Sym^d(\bV_1)$. Because $M_n=0$ for $n<3$ and $\dim(M_3)=1$ it follows that $M$ would contain $N=\Sym^3(\bV_1)$ as a summand. However, $\dim M_n$ has linear growth, while $\dim N_n$ has quadratic growth, so we have a contradiction.

This shows that the projectives of $\Rep^{\rm pol}(\fh_1)$ are not closed under tensor product, and that the hypothesis of Proposition~\ref{prop:gendeg2} that $\ell(\lambda) < n$ is necessary. 
\end{example}

\subsection{Adjoints of specialization} \label{ss:adjoints}

The specialization functor $\Gamma_n$ is co-continuous since it is exact and commutes with direct sums. Indeed, direct sums in $\Rep^{\pol}(\fh)$ and $\Rep^{\pol}(\fh_n)$ are just the usual direct sum representations, and formation of invariants commutes with direct sums. It follows that $\Gamma_n$ admits a right adjoint $\Sigma_n$. We now examine it.

\begin{proposition} \label{prop:right-adjoint}
Let $M$ be a polynomial $\fh_n$-module.
\begin{enumerate}
\item We have a decomposition
\begin{displaymath}
\Sigma_n(M) = \bigoplus_{\mu} \Hom_{\fh_n}(\bS_{\mu}(\bV_n), M) \otimes \bS_{\mu}(\bW)
\end{displaymath}
of $\fgl$-modules, where the sum is over all partitions $\mu$.
\item The co-unit $\Gamma_n(\Sigma_n(M)) \to M$ is an isomorphism.
\end{enumerate}
\end{proposition}

\begin{proof}
Let $\Sigma_n(M)=\bigoplus_{\mu} U_{\mu} \otimes \bS_{\mu}(\bW)$ be the decomposition of $\Sigma_n(M)$ into irreducible $\fgl$ representations, where $U_{\mu}$ is a multiplicity space.  We have
\begin{align*}
U_{\mu}
&= \Hom_{\fgl}(\bS_{\mu}(\bW), \Sigma_n(M)) \\
&= \Hom_{\fh}(\bS_{\mu}(\bV), \Sigma_n(M)) \\
&= \Hom_{\fh_n}(\bS_{\mu}(\bV_n), M).
\end{align*}
In the first step, we used the definition of multiplicity space; in the second, Proposition~\ref{prop:gproj2}; and in the third, adjunction. This proves (a).

By (a), we have a decomposition
\begin{displaymath}
\Gamma_n(\Sigma_n(M)) = \bigoplus_{\mu} \Hom_{\fh_n}(\bS_{\mu}(\bV_n), M) \otimes \bS_{\mu}(\bW_n).
\end{displaymath}
In this sum, it suffices to consider partitions $\mu$ with $\ell(\mu) \le n$, since otherwise $\bS_{\mu}(\bW_n)$ vanishes. By Proposition~\ref{prop:gendeg2}(b), the above is exactly the decomposition of $M$ into its $\fgl_n$-isotypic pieces. We thus have a natural isomorphism $\Gamma_n(\Sigma_n(M))=M$, which one verifies is in fact the co-unit of the adjunction.
\end{proof}

The specialization functor $\Gamma_n$ is also continuous, though this is less obvious; we now give a proof.

\begin{proposition}
The specialization functor $\Gamma_n$ is continuous.
\end{proposition}

\begin{proof}
It suffices to prove that $\Gamma_n$ preserves products. Let
\begin{displaymath}
\Phi \colon \Rep^{\pol}(\fh) \to \Rep^{\pol}(\fgl), \qquad
\Phi_n \colon \Rep^{\pol}(\fh_n) \to \Rep^{\pol}(\fgl_n)
\end{displaymath}
be the forgetful functors. Also, let
\begin{displaymath}
\Psi \colon \Rep^{\pol}(\fgl) \to \Rep^{\pol}(\fh), \qquad
\Psi_n \colon \Rep^{\pol}(\fgl_n) \to \Rep^{\pol}(\fh_n)
\end{displaymath}
be the co-continuous functors satisfying
\begin{displaymath}
\Psi(\bS_{\lambda}(\bW))=\bS_{\lambda}(\bV), \qquad
\Psi_n(\bS_{\mu}(\bW_n))=\bS_{\mu}(\bV_n),
\end{displaymath}
where $\ell(\mu) \le n$. By Proposition~\ref{prop:gproj2}, we see that $\Phi$ is the right adjoint of $\Psi$, and therefore preserves products. Similarly, Proposition~\ref{prop:gendeg2} shows that $\Phi_n$ is the right adjoint of $\Psi_n$, and so preserves products. Now consider the following commutative diagram
\begin{displaymath}
\xymatrix@C=4em{
\Rep^{\pol}(\fh) \ar[r]^{\Gamma_n} \ar[d]_{\Phi} & \Rep^{\pol}(\fh_n) \ar[d]^{\Phi_n} \\
\Rep^{\pol}(\fgl) \ar[r] & \Rep^{\pol}(\fgl_n) }
\end{displaymath}
where the bottom map is the specialization functor on $\fgl$ representations. The bottom map preserves products since products in these categories are computed by taking products of multiplicity spaces. Since $\Phi_n$ is a conservative functor, it follows that $\Gamma_n$ preserves products.
\end{proof}

From the above, it follows that $\Gamma_n$ admits a left adjoint $\Delta_n$. We now examine it.

\begin{proposition}
We have the following:
\begin{enumerate}
\item Let $\lambda$ be a partition with $\ell(\lambda) \le n$. Then $\Delta_n(\bS_{\lambda}(\bV_n))=\bS_{\lambda}(\bV)$.
\item Let $M$ be a polynomial $\fh_n$-module. Write $M=\coker(f)$ where $f \colon P_1 \to P_0$ is a map of modules that are direct sums of $\bS_{\lambda}(\bV_n)$ with $\ell(\lambda) \le n$. Let $g \colon Q_1 \to Q_0$ be the corresponding map of sums of $\bS_{\lambda}(\bV)$'s. Then $\Delta_n(M)=\coker(g)$.
\item The unit $M \to \Gamma_n(\Delta_n(M))$ is an isomorphism for any $M \in \Rep^{\pol}(\fh_n)$.
\end{enumerate}
\end{proposition}

\begin{proof}
(a) For $N \in \Rep^{\pol}(\fh)$ we have
\begin{align*}
\Hom_{\fh}(\Delta_n(\bS_{\lambda}(\bV_n)), N)
&= \Hom_{\fh_n}(\bS_{\lambda}(\bV_n), \Gamma_n(N))
= \Hom_{\fgl_n}(\bS_{\lambda}(\bW_n), \Gamma_n(N)) \\
&= \Hom_{\fgl}(\bS_{\lambda}(\bW), N)
= \Hom_{\fh}(\bS_{\lambda}(\bV), N),
\end{align*}
and so $\Delta_n(\bS_{\lambda}(\bV_n)) = \bS_{\lambda}(\bV)$ by Yoneda's lemma. In the first step above we used adjunction; in the second step, Proposition~\ref{prop:gendeg2}(b); in the third, the fact that $\bS_{\lambda}(\bW_n)$ is irreducible; and in the final step Proposition~\ref{prop:gproj2}.

(b) follows from (a) and the fact that $\Delta_n$ is right-exact.

(c) follows from (b) and the fact that $\Gamma_n$ is exact.
\end{proof}

\subsection{Serre quotients}

We now show how $\Rep^{\pol}(\fh_n)$ can be obtained directly from $\Rep^{\pol}(\fh)$ as a Serre quotient. Let $\cK_n$ be the full subcategory of $\Rep^{\pol}(\fh)$ spanned by modules $M$ whose simple $\fgl$-factors have the form $\bS_{\lambda}(\bW)$ with $\ell(\lambda)>n$. This category is exactly the kernel of the functor $\Gamma_n$, since $\bS_{\lambda}(\bW_n)$ vanishes if and only if $\ell(\lambda)>n$. Since $\Gamma_n$ is exact, it therefore induces a functor
\begin{equation} \label{eq:serre}
\Rep^{\pol}(\fh)/\cK_n \to \Rep^{\pol}(\fh_n)
\end{equation}
by the mapping property for Serre quotients. The following is our main result concerning this functor.

\begin{proposition} \label{prop:serre}
The functor \eqref{eq:serre} is an equivalence.
\end{proposition}

\begin{proof}
This follows from \cite[Ch.~III, Prop.~5]{gabriel}, as $\Gamma_n$ has a right adjoint $\Sigma_n$ and the co-unit $\Gamma_n \circ \Sigma_n \to \id$ is an isomorphism (Proposition~\ref{prop:right-adjoint}).
\end{proof}

As a corollary, we can transfer noetherianity from $\fh$ to $\fh_n$:

\begin{corollary}
If $\Rep^{\pol}(\fh)$ is locally noetherian then so is $\Rep^{\pol}(\fh_n)$.
\end{corollary}

\begin{proof}
Any quotient of a locally noetherian category by a localizing subcategory is locally noetherian \cite[Ch.~III, Prop.~9, Cor.~1]{gabriel}.
\end{proof}

In particular, when $P$ is $\Com$ or $\Com^{\rn\ru}$, we have seen that $\Rep^{\pol}(\fh)$ is locally noetherian (Corollary~\ref{cor:comm-noeth}), and so $\Rep^{\pol}(\fh_n)$ is also locally noetherian. This proves Theorem~\ref{mainthm1}.

\subsection{An example} \label{ss:finite-ex}
Consider $\fW_1^+$ as a $\fW_1^+$ representation. This representation is not polynomial.  Indeed, it is spanned by $\{L_n\}_{n \geq 0}$ with $L_n := x^{n+1} \del_x$ in degree $n$. The element $v:=L_2$  generates the degree $\geq 2$ part of $\fW_1^+$, because $L_1^i v$ is proportional to $L_{2 + i}$.  If $\fW_1^+$ were polynomial then by Proposition \ref{prop:gendeg2} there would be a unique map $\Sym^2 \bV_1  \to \fW_1^+$ taking the distinguished generator $e$ of $\Sym^2 \bV_1$ to  $v$, and this map would be surjective.   The degree $6$ component of $\Sym^2 \bV_1$ is $3$ dimensional, spanned by $L_3 L_1 e, L_2 L_2 e,$ and $L_1 L_1 L_2 e$.  It follows from the relations $$L_3 L_1 v = L_2 L_2 v = L_1 L_1 L_2 v = 0,$$ that any map taking $e$ to $v$ is the zero map in degree $6$. Therefore no such map exists.  
	
It follows that $\fW_1$ is not polynomial as a $\fW_1$ representation. (If it were, then the restricted $\fW_1^+$ sub-representation  $\fW_1^+ \subseteq \fW_1$ would be polynomial).  We note that $\fW_1$ is isomorphic to $$\Hom_{\bk[x]}(\Omega^1_{\bk[x]}, \bk[x]) \subseteq (\Omega^1_{\bk[x]})^{\vee} \otimes \bk[x]$$ so it is contained in the tensor category generated by $\bV_1$ and $\bV_1^{\vee}$.
\section{Hilbert series} \label{s:hilbert}

In this final section, we study the Hilbert series and formal characters of polynomial representations of the Witt Lie algebra. The proofs rely on fine results concerning $\FS^{\op}$-modules established in \cite{To}, which is why we cannot treat general operads.

\subsection{Definitions}

Let $M$ be a polynomial representation of $\fgl_n$. Then $M$ decomposes into a direct sum of simple representations $\bS_{\lambda}(\bW_n)$, with $\ell(\lambda) \le n$. We assume that each simple representation has finite multiplicity. This decomposition induces a grading of $M$, where the degree $d$ piece is the sum of the simple representations with $\vert \lambda \vert=d$.

\begin{definition}
Let $m_d=\dim(M_d)$. The \defn{Hilbert series of $M$}, denoted $\rH_M(t)$, is the series $\sum_{d \ge 0} m_d t^d$. It belongs to $\bZ \lbb t \rbb$.
\end{definition}

We can also define a finer invariant that records the multiplicities in $M$. Let $\Lambda_n$ be the ring of symmetric polynomials in $n$ variables $x_1, \ldots, x_n$ with coefficients in $\bQ$. As a vector space, this has a basis given by Schur polynomials $s_{\lambda}(x_1,\dots,x_n)$ with $\ell(\lambda) \le n$. Let $\widehat{\Lambda}_n$ be the completion of this ring at the ideal of positive degree elements. An element of $\widehat{\Lambda}_n$ can be expressed as a formal sum $\sum_{\lambda} m_{\lambda} s_{\lambda}(x_1,\dots,x_n)$ with $m_{\lambda} \in \bQ$; it can also be expressed as a power series in the $x_i$ variables.

\begin{definition}
For $\ell(\lambda) \le n$, let $m_{\lambda}$ be the multiplicity of $\bS_{\lambda}(\bW_n)$ in $M$. The \defn{formal character} of $M$, denoted $\ch(M)$, is $\sum_{\lambda} m_{\lambda} s_{\lambda}(x_1,\dots,x_n)$. It is an element of $\widehat{\Lambda}_n$.
\end{definition}

There is a map $\widehat{\Lambda}_n \to \bQ\lbb t \rbb$ induced by $x_i \mapsto t$ for all $1 \le i \le n$. Under this map, the formal character of $M$ specializes to its Hilbert series.

Suppose now that we are in the setting of \S \ref{ss:setup}. If $M$ is a polynomial $\fh_n$-module then it is also a polynomial $\fgl_n$-module, and so the above definitions apply, assuming each simple $\fgl_n$-module has finite multiplicity in $M$. Suppose that each Schur functor has finite multiplicity in $P$ (which is the case for all examples in \S \ref{ss:opex}). Then every simple has finite multiplicity in $\bV_n=P\{\bW_n\}$, and thus in all of its tensor powers. Thus the same is true for any finitely generated polynomial $\fh_n$-module by Proposition~\ref{prop:gendeg2}(e), and so Hilbert series and formal characters are defined for these modules.

\subsection{Statement of results}

The following results are our main results on Hilbert series. In what follows, $M$ denotes a finitely generated polynomial $\fW_n^{\pos}$-module. We begin with results on Hilbert series.

\begin{theorem} \label{thm:hilbert}
The Hilbert series of $M$ is a rational function of $t$ whose denominator is a product of polynomials of the form $1-t^m$.
\end{theorem}

\begin{corollary}
The function $n \mapsto \dim(M_n)$ is a quasi-polynomial of $n$, for $n \gg 0$.
\end{corollary}

We now turn to the formal character:

\begin{theorem}\label{thm:formalcharacter}
The formal character of $M$ is a rational function of the variables $x_i$, whose denominator is a product of polynomials of the form $1-x_i^m$. Moreover, if $M$ is a subquotient of a representation generated in degrees $\le d$ then the exponents $m$ satisfy $m \le d$.
\end{theorem}

Theorem~\ref{thm:hilbert} follows from Theorem~\ref{thm:formalcharacter} by specializing the $x_i$'s to $t$. Thus it suffices to prove the latter theorem, and this is what we do in the remainder of this section.

\begin{remark}
The above results also apply if $M$ is a finitely generated polynomial $\fW_n$-module, since we can simply restrict to $\fW_n^{\pos}$. In particular, they yield Theorem~\ref{mainthm2}.
\end{remark}

\subsection{A reduction}

Let $\Lambda$ be the ring of symmetric functions in variables $\{x_i\}_{i \ge 1}$, and let $\widehat{\Lambda}$ be its completion at its ideal of positive degree elements. If $M$ is a polynomial $\fgl$-module in which all Schur functors $\bS_\lambda$ have finite multiplicity $m_\lambda$, then we can define its formal character
\[
  \ch(M) = \sum_\lambda m_\lambda s_\lambda
\]
as an element of $\widehat{\Lambda}$, where now $s_\lambda$ is the Schur function in $\{x_i\}$.

Now suppose that $N$ is an $\FB$-module such that $N([n])$ is finite dimensional for all $n$. Letting $m_{\lambda}$ be the multiplicity of the Specht module $S^{\lambda}$ in $N([n])$ (where $n=\vert \lambda \vert$), we define the formal character of $N$ to be $\sum_{\lambda} m_{\lambda} s_{\lambda}$. These two notions of formal character are compatible with Schur--Weyl duality: if $M \in \Rep^{\pol}(\fgl)$ corresponds to $N \in \Rep(\FB)$ then $\ch(M)=\ch(N)$.

There is a specialization homomorphism
\begin{displaymath}
\pi_n \colon \widehat{\Lambda} \to \widehat{\Lambda}_n
\end{displaymath}
induced by the rule $\pi_n(x_i)=x_i$ for $1 \le i \le n$ and $\pi_n(x_i)=0$ for $i>n$. This maps the symmetric function $s_{\lambda}$ to the symmetric polynomial $s_{\lambda}(x_1,\dots,x_n)$ if $\ell(\lambda) \le n$, and to~0 otherwise. This homomorphism is compatible with the specialization functor $\Gamma_n$; that is, if $M$ is a finitely generated polynomial $\fW^{\pos}$-module then
\begin{displaymath}
\pi_n(\ch(M)) = \ch(\Gamma_n(M)).
\end{displaymath}
Indeed, everything above simply depends on the underlying $\fgl$-module, and can therefore be checked in the case $M=\bS_{\lambda}(\bW)$, where it is clear.

In the remainder of this section, we prove the following theorem:

\begin{theorem} \label{thm:FSopRationality}
Let $M$ be a finitely generated $\FS^\op$-module. Then for every $n$, the specialization $\pi_n(\ch(M))$ is a rational function of the $x_i$'s whose denominator is a product of factors $1-x_i^m$. Moreover, if $M$ is a subquotient of an $\FS^{\op}$-module generated in degrees $\le d$ then the exponents $m$ satisfy $m \le d$.
\end{theorem}

We now show that this is sufficient:

\begin{lemma} \label{lem:formalcharacter}
Theorem~\ref{thm:FSopRationality} implies Theorem~\ref{thm:formalcharacter}.
\end{lemma}

\begin{proof}
  Let $M$ be a given finitely generated polynomial $\fW^{\pos}_n$-module. By Proposition~\ref{prop:ess-surj}, $M$ is in the essential image of the specialization functor $\Gamma_n$. The proof of the corollary shows that we have $M \cong \Gamma_n(M')$ for some \emph{finitely generated} polynomial $\fW^{\pos}$-module $M'$; even better, if $M$ is a subquotient of a $\fW^{\pos}_n$-module generated in degrees $\le d$ then we can take the same to be true for $M'$ . Indeed,  if $M \cong K/J$ for
  \[
    J \subseteq K \subseteq  F = \bigoplus_{i = 1}^r \bS_{\lambda_i}(\bV_n)
  \]
  with $\ell(\lambda_i) \leq n$ and $|\lambda_i| \leq d$ then we may take
  \[
    M' = \frac{\im(\Delta_n(K) \to \Delta_n(F))}{\im(\Delta_n(J) \to \Delta_n(F))},
  \]
  since by exactness of $\Gamma_n$  we have
  \[
    \Gamma_n\left(\frac{\im(\Delta_n(K) \to \Delta_n(F))}{\im(\Delta_n(J) \to \Delta_n(F))}\right) \cong
    \frac{ \im(\Gamma_n \Delta_n(K) \to \Gamma_n\Delta_n(F)) }{ \im(\Gamma_n \Delta_n(J) \to \Gamma_n \Delta_n(F))} \cong K/J.
  \]
  Let $M''$ be the $\FS^{\op}$-module corresponding to $M'$ under the equivalence in Theorem~\ref{thm:equiv}. We then have
\begin{displaymath}
\ch(M) = \pi_n(\ch(M')) = \pi_n(\ch(M'')),
\end{displaymath}
where in the first step we used the compatibility of formal characters with specialization, and in the second the fact that $M'$ and $M''$ have the same formal character. Theorem~\ref{thm:FSopRationality} now shows that $\ch(M)$ has the required form.
\end{proof}

\subsection{Review of \cite{To}}

To prove Theorem~\ref{thm:FSopRationality}, we make use of the results from \cite{To} on characters of $\FS^{\op}$-modules, which we now review (see \cite[\S11]{To} for a more detailed summary). We begin by introducing some subspaces of $\widehat{\Lambda}$.

\begin{definition}
Let $\cF_{\leq k}$ be the $\bQ$-subspace of $\widehat{\Lambda}$ consisting of elements of the form
\begin{displaymath}
\sum_{\ell(\lambda) \le k} c_\lambda  s_\lambda,
\end{displaymath}
where $\lambda$ is a partition, $c_\lambda \in \bQ$, and $s_{\lambda}$ is the Schur function associated to $\lambda$.           
\end{definition}

\begin{definition}
Let $A$ be a partition and let $m_i(A)$ be the multiplicity of $i$ in $A$. Let $p_d = \sum_i x_i^d$ be the power sum and for a partition $\lambda$ let $p_\lambda = \prod_{i \geq 1} p_i^{m_i(\lambda)}$ denote the power sum symmetric function. We define
\begin{displaymath}
e^{A} := \exp\left( \sum_{i \geq 1} m_i(A) \sum_{n \geq 1}  \frac{p_{ni}}{n}\right) =   \prod_{i \geq 1} \prod_{\alpha} \left(\frac{1}{1- x_\alpha^{i}} \right)^{m_i(A)}
\end{displaymath}
and let $\cV_{A,r} \subseteq \widehat \Lambda$ be the subspace of symmetric functions which can be written as formal sums of the form $$\sum_{\ell(\lambda) < r } c_\lambda p_\lambda e^{A}$$ where $c_\lambda \in \bQ$.
\end{definition}

\begin{remark}
 	We use the notation $e^A$ because under the specialization homomorphism $p_1 \mapsto t$ and $p_i \mapsto 0$ for $i > 1$,  the symmetric function $e^A$ maps to an exponential.  In \cite{To} the class function on symmetric groups corresponding to $e^A$ was referred to as a \emph{character exponential}, in analogy to the classical  notion of a character polynomial.  We have the alternate expression for $e^A$  $$e^A = p_A\left[\sum_{m \geq 0} h_m\right],$$ where the square brackets denotes the operation of plethysm.  
\end{remark}

The following is the main theorem we will need.

\begin{theorem}[{\cite[Theorems 1.2, 11.1]{To}}] \label{thm:tosteson}
Let $M$ be a finitely generated $\FS^{\op}$-module that is a subquotient of one generated in degrees $\le d$. Then $\ch(M)$ is contained in the sum of the spaces
\begin{displaymath}
\cV_{A_i,r_i} \cap \cF_{\leq k}, \qquad i = 1, \dots, N
\end{displaymath}
for a finite list of pairs $(A_i, r_i)$  such that $|A_i| \leq d$.
\end{theorem}

Fix numbers $r$ and $k \geq r$, and a partition $A$. We write $a_i := m_i(A)$ for the multiplicity of $i$ in $A$. Neither $\cF_{\leq k}$ nor $\cV_{A,r}$ is finite dimensional, but their intersection is. This is what gives the above theorem its power. We now discuss the structure of this intersection in more detail. For this, we use the following elements (which depend on our choice of $A$).    

\begin{definition}
  For $n \in \bN$, define $u_n := p_n - \sum_{i|n} a_i$. For a partition $\lambda$, let 
  \[
    {u_\lambda} := \prod_{i \geq 1}  {u_i^{m_i(\lambda)}}. \qedhere
  \]
  We also define
  \[
  	\lambda! :=\prod_{i \geq 1} m_i(\lambda)!~ \qquad z_\lambda := \lambda!\prod_{i \geq1} {i^{m_i(\lambda)} }. 
  \]
\end{definition}

The elements $u_\lambda$ are important because of the following identity. Recall that the Hall inner product is the unique bilinear form $\langle\ , \ \rangle \colon \Lambda \otimes \widehat \Lambda  \to \bQ$ satisfying $\langle s_\lambda, s_\mu \rangle = \delta_{\lambda,\mu}$ for all partitions $\lambda, \mu$.

\begin{proposition}[{\cite[Lemma 10.10]{To}}]
 For $s \in \cV_{A,r} \cap \cF_{\leq k}$,  we have that
 \addtocounter{equation}{-1}
 \begin{subequations}
 \begin{equation}\label{expansion}
s = \sum_{\lambda} \left\langle \frac{u_\lambda}{z_\lambda}, s \right\rangle  p_\lambda  e^A, 
\end{equation}
\end{subequations}
where the sum is over all partitions $\lambda$.
\end{proposition}

To compute the inner products $\langle u_\lambda,s\rangle$, we use the following ring. 

\begin{definition}
  Let $R_{r,k, A}$ be the ring $$R_{r,k,A} := \frac{\Lambda}{(u_1, u_2, \dots)^r + (e_{k+1}, e_{k+2}, \dots) },$$ where $e_k$ is the elementary symmetric function.
\end{definition}

In \cite[Proposition 10.11]{To}, it is shown that for  $s \in \cV_{A,r} \cap \cF_{\leq k}$, the value of $\langle f,s\rangle$  only depends on the image of $f$ in $R_{r,k,A}$. In fact, the Hall inner product induces an isomorphism $\cV_{A,r} \cap \cF_{\leq k} \cong (R_{r,k,A})^*$. There is a collection of elements $E_m \in \widehat{\Lambda}$, $m \geq 1$ \cite[Definition 11.7]{To} such that the canonical projection induces an isomorphism
\[
  \bQ[E_1, \dots, E_k]/(E_1, \dots, E_k)^r \cong R_{r,k,A},
\]
see \cite[Proposition 11.11]{To}.   For a partition $\nu$ we define $E_\nu = \prod_{i} E_i^{m_i(\nu)}.$   Then the images of $\{E_\nu\}_{\nu \in {\rm Part}(r,k)}$ in $R_{r,k,A}$ form a basis of the maximal ideal  $(E_1, \dots, E_k) \subseteq R_{r,k,A}$.  Here ${\rm Part}(r,k)$ is the set of non-empty partitions whose Young diagrams have $\leq k$ columns and $< r$ rows (we emphasize here that the first inequality is weak and the second is strict). 

\subsection{Proof of Theorem~\ref{thm:FSopRationality}}

We now have the necessary tools to prove the theorem. By Theorem~\ref{thm:tosteson}, Theorem~\ref{thm:FSopRationality} follows from the following theorem:

\begin{theorem} \label{componentwiserationality}
Let $r$ and $k$ be non-negative integers, let $A$ be a partition, and let $s$ be an element of $\cV_{A,r} \cap \cF_{\leq k}$. Then $\pi_n(s)$ is a rational function of the $x_i$'s whose denominator is a product of factors of the form $1-x_i^m$ with $m \le \vert A \vert$.
\end{theorem}

\begin{proof}
By \cite[Equation 11.9]{To}, we have the following identity in $R_{r,k,A}[\![t]\!]$:
\begin{equation*}
\sum_{n \geq 1} \frac{u_n}{n} t^n =  \sum_{\nu \in {\rm Part}(r,k) } E_\nu ~t^{|\nu|} \frac{   (\ell(\nu) -  1)!}{\nu! \left (\prod _{i} (1 - t^i)^{a_i} \right)^{\ell(\nu)}}.
\end{equation*}  
We apply the linear transformation $R_{r,k,A}[\![t]\!] \to  R_{r,k,A} \otimes \widehat{\Lambda}$ induced by $t^n \mapsto p_n$ to this identity. Notice that under this transformation we have
$$
\frac{t^{|\nu|}}{\prod_{i}(1 -t^i)^{a_i \ell(\nu)}} \mapsto \sum_{\alpha}  \frac{x_\alpha^{|\nu|}}{\prod_{i}(1 -x_\alpha^i)^{a_i \ell(\nu)}}.
$$
Taking exponentials, we obtain the identity 
\begin{equation*}
  \sum_{\lambda} \frac{p_\lambda u_\lambda}{z_\lambda}  = \exp\left( \sum_{n} \frac{u_n p_n}{n} \right) = \exp \left(  \sum_{\nu \in {\rm Part}(r,k)}  \frac{E_\nu}{\nu!} \sum_{\alpha}  \frac{x_\alpha^{|\nu|} (\ell(\nu) - 1)!}{\prod_{i}(1 -x_\alpha^i)^{a_i \ell(\nu)}} \right)  
\end{equation*}
in $R_{r,k,A} \otimes \widehat{\Lambda}$.  Each $E_\nu$ is nilpotent of order $r$ \cite[Proposition 11.11]{To}.  So expanding out the right hand side exponential,  we obtain a sum over $\nu \in {\rm Part}(r,k)$  of $E_\nu$ times a coefficient in $\widehat{\Lambda}$. The coefficient of $E_\nu$ is a finite linear combination of functions of the form
$$
\prod_{j =1}^N  \left( \sum_{\alpha} \frac{  x_\alpha^{|\nu_j|}}{\prod_{i}(1 -x_\alpha^i)^{a_i \ell(\nu_j)}}\right)
$$
where $\nu = \sum_{j = 1}^N \nu_j$ in the sense that $m_i(\nu) = \sum_{j = 1}^N m_i(\nu_j)$ for every positive integer $i$. Finally, we apply the linear transformation
\[
  \langle -, s \rangle e^A \colon   R_{r,k,A} \otimes \widehat \Lambda \to \widehat \Lambda, \qquad f \otimes g \mapsto \langle f, s \rangle ge^A.
\]
We see from \eqref{expansion} that $s = \langle \sum_{\lambda} \frac{p_\lambda u_\lambda}{z_\lambda}, s \rangle e^A$   is a linear combination of symmetric functions of the form
\[
  \frac{1}{\prod_{\alpha} \prod_{i} (1-x_\alpha^i)^{a_i}}\prod_{j = 1}^N  \sum_{\alpha} \frac{x_{\alpha}^{m_j}}{\prod_{i}( 1-x_\alpha^i)^{a_i r_j}},
\]
where $\sum r_j < r$  and $\sum_{j} m_j \leq r k$.  
  In particular $s$ specializes to a rational function with denominator $\prod_{\alpha = 1}^d  \prod_{i}  (1 - x_\alpha^i)^{a_i r}.$   
\end{proof}

\end{document}